\documentclass{article}
\usepackage{amsmath,amssymb,amsfonts,amsthm}
\usepackage{geometry}
\usepackage{enumitem}
\usepackage{url}
\usepackage{comment}
\newtheorem{proposition}{Proposition}
\newtheorem{theorem}{Theorem}[section]
\newtheorem{definition}[theorem]{Definition}
\newtheorem{lemma}[theorem]{Lemma}

\newtheorem{Example}[theorem]{Example}

\title{\textbf{Haar Measure on Fuzzy Lie Group}}
\author{
     S. S. Sangodele, M. E. Egwe$^*$ \\
    Department of Mathematics, University of Ibadan, Ibadan, Nigeria \\
    \texttt{Email:murphy.egwe@ui.edu.ng$^*$, ss.sangodele@ui.edu.ng}
}
\date\today

\begin{document}

\maketitle

\begin{abstract}
Let $\mathfrak{G}$ be a locally compact Lie group and \(d\mu\) a Haar measure defined on $\mathfrak{G}$. We consider the existence of a fuzzy analogue of $\mathfrak{G}$ denoted by $\mathfrak{G_f}$  called a fuzzy Lie group and develop a fuzzy Haar measure \(\mu_f\) on $\mathfrak{G_f}$. Also we construct a fuzzy Haar integral with respect to the corresponding fuzzy Haar measure on $\mathfrak{G_f}$. And finally we show that there exists a fuzzy Haar integral that is unique up to a multiplicative constant.\\
\ \\
\textbf{Key words:} Haar measure,  Fuzzy Lie group, Fuzzy manifolds\\
\textbf{MSC(2020):} , 46S40, 54A40, 20N25, 43A90, 28C10\\
\end{abstract}

\section{Introduction}
The fuzzy measure theory was first introduced by Sugeno in [9],[10], where the classical measure theory was generalised. This resulted in the creation of fuzzy measure theory  by considering the fuzzy analogue of measure theory of classical set theory. In this case, the results were obtained by introducing monotonicity and semicontinuity to replace the additivity requirement of classical measure theory. Measurable functions defined on a fuzzy measure space were introduced by Zhenyuan Wang in [12]. Fuzzy Lie algebra recently took a centre stage when Akram [1] introduced his book in 2018. The exponential map of any fuzzy Lie algebra into a fuzzy Lie group remains an open problem. The approach in this work is the manifold property of fuzzy Lie group.  The work of Alfred Haar in 1933 gave birth to the concept of Haar measure by establishing an invariant measure property on the action of a topological group. In [2] Alfsen constructed the proof of existence and uniqueness of Haar measure on Lie groups. Our aim in this series is to establish some results on integration and integral operators on fuzzy Lie groups with an underlying measure of integration.  In  this particular paper, we develop  a fuzzy Haar measure by constructing a fuzzy analogue of Haar measure where the properties of monotonicity, continuity and invariance are profusely required. Finally we shall introduce and discuss extensively the concept of fuzzy Haar integral and show its existence up to a multiplicative constant. The main results are Theorems 3.23, 8.3, 9.1 and 9.4.

\section{Fuzzy Set}
\subsection*{}

\begin{definition} [Fuzzy Set] [7] \\
    Given a set $X$. A fuzzy set $\mathfrak{A}$ in $X$ is defined as a set of ordered pairs:
    \[
    \mathfrak{A} = \{(x, \quad \mathfrak{P}_{\mathfrak{A}} (x)) \quad| \quad x \in X\}
    \]
    where
    \[
    \mathfrak{P}_{\mathfrak{A}} : X \rightarrow [0,1].
    \]
    such that,
    \[
    \mathfrak{P}_{\mathfrak{A}}(x) =
    \begin{cases}
        1, & if \ x \in \mathfrak{A} \\
        (0,1), & if \ x \ \text{is partially in} \ \mathfrak{A} \\
        0, & if \ x \notin \mathfrak{A}
    \end{cases}
    \]
    And $\mathfrak{P}_{\mathfrak{A}} (x)$ denotes the membership degree of elements in $X$.
\end{definition}

\begin{definition}[Fuzzy subset][1]\\
Let \(\{U_i\}_{i\in I}\) be a collection of fuzzy subsets of \(X\). Then we define the fuzzy subsets \(\displaystyle\bigcap_{i\in I}U_i\) and \(\displaystyle\bigcup_{i\in I}U_i\) by
\[
\left(\displaystyle\bigcap_{i\in I}U_i\right)(x) = \inf_{i\in I}\{U_i(x)\} \quad \text{for all } x\in X.
\]
\[
\left(\displaystyle\bigcup_{i\in I}U_i\right)(x) = \sup_{i\in I}\{U_i(x)\} \quad \text{for all } x\in X.
\]
\end{definition}

\begin{definition}[t-cut set][1]\\
Let \(U\) be a fuzzy set on a universe of discourse \(X\). For any \(t\in [0,1]\), a \(t\)-cut set (t-level set) of \(U\) is defined as
\[
U_t = \{x\in X \mid U(x)\ge t\}.
\]
The set
\[
U_t^+ = \{x\in X \mid U(x) > t\}
\]
is called the strong \(t\)-cut set of \(U\).
\end{definition}

\begin{definition}[Fuzzy relation][1]\\
A fuzzy relation on a nonempty set \(X\) is a fuzzy operation \(\gamma: X\times X \to [0,1]\). Let \(U\) be a fuzzy set on \(X\) and \(\gamma\) a fuzzy relation on \(X\). We call \(\gamma\) a fuzzy relation on \(U\) if
\[
\gamma(x,y) \le \min\{U(x),U(y)\} \quad \text{for all } x,y\in X.
\]
\end{definition}

\begin{definition}[][1]\\
Let \(\gamma\) be a fuzzy relation on a universe \(X\). For any \(t\in [0,1]\), a \(t\)-cut set (t-level set) of \(\gamma\) is denoted by
\[
U(\gamma,t) = \{(x,y)\in X\times X \mid \gamma(x,y)\ge t\}.
\]
\end{definition}

\begin{definition}[Fuzzy equivalence relation][1]\\
Let \(\gamma\) be a fuzzy relation on \(X\). Then \(\gamma\) is called a fuzzy equivalence relation on \(X\) if the following conditions are satisfied:
\begin{enumerate}
\item[(a)] \(\gamma\) is fuzzy reflexive, i.e. \(\gamma(x,x)=0\) for each \(x\in X\);
\item[(b)] \(\gamma\) is fuzzy symmetric, i.e. \(\gamma(x,y)=\gamma(y,x)\) for each \(x,y\in X\);
\item[(c)] \(\gamma\) is fuzzy transitive, i.e. \(\gamma(x,z)\ge (\gamma(x,y)\wedge \gamma(y,z))\).
\end{enumerate}
\end{definition}

\begin{definition}(Fuzzy function)
\\
Let \(f\) be a mapping defined on a set \(X\). If \(U\) is a fuzzy set in \(X\), then the fuzzy set \(V\) in \(f(X)\) defined by
\[
V(y) = \sup_{x\in f^{-1}(y)} U(x) \quad \text{for all } y\in f(X)
\]
is called the image of \(U\) under \(f\). If \(V\) is a fuzzy set on \(f(X)\), then the fuzzy set \(U = V\circ f\) in \(X\), i.e. the fuzzy set defined by \(U(x)=V(f(x))\) for all \(x\in X\), is called the pre-image of \(V\) under \(f\).
\end{definition}

\begin{definition} (Fuzzy Inverse function)\\
Let \(f\) be a mapping from a set \(X\) into a set \(Y\).
\begin{enumerate}
\item[(1)] Let \(U\) be a fuzzy set in \(Y\). The inverse image of \(U\), denoted by \(f^{-1}(U)\), is the fuzzy set in \(X\) defined by \(f^{-1}(U)(x)=U(f(x))\) for all \(x\in X\).
\item[(2)] Let \(V\) be a fuzzy set in \(X\). The image of \(V\) is denoted by \(f(V)\) and is the fuzzy set in \(Y\) such that
\[
f(V)(y) = \sup_{x\in f^{-1}(y)} V(x), \quad if \ f^{-1}(y) \neq \emptyset, \quad \text{for all } y\in Y.
\]
\end{enumerate}
\end{definition}

\subsection{Fuzzy Vector Spaces}

\begin{definition}[][3]\\
Let \(E\) be a vector space and \(\mathfrak{A}_1,\dots,\mathfrak{A}_n\) be fuzzy sets in \(E\). We define \(\mathfrak{A}_1\times\cdots\times \mathfrak{A}_n\) to be the fuzzy set \(\mathfrak{A}\) in \(E^n\) whose membership function is given by
\[
\nu_\mathfrak{A} (x_1,\dots,x_n) = \min\{\nu_{\mathfrak{A}_1}(x_1),\dots,\nu_{\mathfrak{A}_n}(x_n)\}.
\]
Let \(f:E^{n} \to E\) be defined by \(f(x_1,\dots ,x_n)=x_1+\dots +x_n\). We define  \(\mathfrak{A}_1+\dots+\mathfrak{A}_n = f(\mathfrak{A})\). For a scalar \(\beta\) and a fuzzy set \(\mathfrak{B}\) in \(E\), we define \(\beta \mathfrak{B} = g(\mathfrak{B})\) where \(g:E\to E\) is given by \(g(x)=  x\).
\end{definition}
The following proposition gives the properties of fuzzy linear map as can be seen in
\begin{proposition}[][7]\\
Let \(E,F\) be vector spaces over \(\mathbb{K}\) and let \(f\) be a linear map from \(E\) to \(F\). Then for all fuzzy sets \(\mathfrak{A},\mathfrak{B}\) in \(E\) and all scalars \(\beta\):
\begin{enumerate}
\item[(a)] \(f(\mathfrak{A}+\mathfrak{B}) = f(\mathfrak{A}) + f(\mathfrak{B})\);
\item[(b)] \(f(\beta \mathfrak{A}) = \beta f(\mathfrak{A})\).
\end{enumerate}
\end{proposition}
\subsection{Topology of Fuzzy Sets}

Let \(X\) be a non-empty set and let \(I=[0,1]\). \(I_X\) will denote the set of all functions \(\gamma: X\to I\). A member of \(I_X\) is called a fuzzy subset of \(X\).

\begin{definition}[Fuzzy topological space][3]\\
Let \(\gamma\) be a fuzzy subset of \(X\). A collection \(\tau\) of fuzzy subsets of \(\gamma\) satisfying
\begin{enumerate}
\item[(i)] \(k\cap\gamma\in\tau\) for all \(k\in I\);
\item[(ii)] \(U_i\in\tau\) for all \(i\in \mathbb{N}\) \(\Rightarrow \bigcup\{U_i: i\in\mathbb{N}\}\in\tau\);
\item[(iii)] \(U,V\in\tau\) \(\Rightarrow U\cap V\in\tau\).
\end{enumerate}
\end{definition}

\noindent
is called a fuzzy topology on \(\gamma\). The pair \(\langle \gamma, \tau \rangle\) is called a fuzzy topological space. Members of \(\tau\) will be called fuzzy open sets and their complements with respect to \(\gamma\) are called fuzzy closed sets of \(\langle \gamma, \tau \rangle\).

If \(\mathcal{B}\) be a given collection of fuzzy subset of \(\gamma\), then the family of all possible unions and finite intersections of the members of \(\mathcal{B}\) and the family \(\{\gamma \cap k: k \in I\}\) is a fuzzy topology on \(\gamma\) and it will be denoted by \(\tau(\mathcal{B})\).

\begin{definition}[Open base][3]\\
\(\mathcal{B} \in \tau\) is called an open base of \(\tau\) if every member of \(\tau\) can be expressed as the union of some members of \(\mathcal{B}\).
\end{definition}

\begin{definition}[Fuzzy hausdorff space][3]\\
A fuzzy topological space \(\langle \gamma, \tau \rangle\) is said to be fuzzy Hausdorff space if \(\forall \ X_p, Y_p \in \gamma(X \times Y)\), \(\exists \ U, V \in \tau\) such that \(X_p \in U, \ Y_p \in V\) and \(U \cap V = \emptyset\).
\end{definition}

\begin{definition}[Fuzzy compact space][3]\\
A fuzzy topological space \(\langle \gamma, \tau \rangle\) is said to be fuzzy compact if \(\forall  \ \beta \in \tau\) satisfying \(\bigcup\{U: U \in \beta\} = \gamma\) and \(\forall \ \varepsilon > 0\), \(\exists\) a finite subcollection \(\beta_0\) of \(\beta\) such that \(\bigcup\{U: U \in \beta_0\} \ge \gamma_\varepsilon\) where \(\gamma_\varepsilon\) is defined by \(\gamma_\varepsilon (x) = g(x) - \varepsilon\) or $0$ according as \(\gamma(x) > \varepsilon\) or \(\gamma(x) \le \varepsilon\).
\end{definition}

\begin{definition}[Fuzzy separability][3]\\
A fuzzy subset \(U\) of \(\gamma\) is said to be fuzzy separated. If \(\exists \ \nu, \delta \in \tau\) such that \(U = \nu \cup \delta\), \(\nu \neq \delta\) and \(\nu \cap \delta = \emptyset\).
\end{definition}

\begin{definition}[Fuzzy connectedness][3]\\
A fuzzy topological space \(\langle \gamma, \tau \rangle\) is said to be fuzzy connected if no fuzzy closed subset of \(\langle \gamma, \tau \rangle\) can be fuzzy separated.
\end{definition}

Let $X$ and $Y$ be a non-empty sets and $\gamma \in I_X, \ U \in I_Y$
\begin{definition}[Fuzzy proper function][3]\\
A fuzzy subset \(Q\) of \(X\times Y\) is said to be a fuzzy proper function from \(\gamma\) to \(U\) if
\begin{enumerate}
    \item $Q\langle x,y \rangle$ $\le \min\{\gamma\langle x \rangle,U\langle y \rangle\} \quad \forall \ \langle x,y\rangle \in X\times Y$,
    \item $\forall \ x\in X,\; \exists \ y_0\in Y \text{ such that } \ Q\langle x,y_0 \rangle=\gamma(x) \ \text{ and } \ Q\langle x,y \rangle=0 \text{ if } \ y\neq y_0$.
\end{enumerate}
\end{definition}

\begin{definition}[Fuzzy proper inverse function][9]\\
Let \(Q:\gamma \to U\) be a proper function. If \(A\le\gamma\), \(B\le U\) then \(Q(A)\) and \(Q^{-1}(B)\) are defined by
\[
\langle Q\langle A \rangle \rangle \langle y \rangle = \sup\{\min\{Q \langle x,y \rangle, \ A \langle x \rangle \}: x\in X\} \quad \forall \ y\in Y,
\]
\[
\langle Q^{-1} \langle B \rangle \rangle \langle x \rangle = \sup\{\min\{Q \langle x,y \rangle, \ B \langle y \rangle \}: y\in Y\} \quad \forall \ x\in X.
\]
\end{definition}

\begin{definition}[Fuzzy bijective function][9]\\
A proper function \(F:\gamma\to U\) is said to be
\begin{enumerate}
\item[(i)] injective if \( Q \langle x_1,y \rangle =\gamma(x_1)\) and \(Q\langle x_2,y \rangle=\gamma(x_2)\) implies \(x_1=x_2\) for all \ \(x_1,x_2\in X\), \(y\in Y\);
\item[(ii)] surjective if \(\forall \ y\in Y\) with \(U(y)\neq0\), \(\exists \ x\in X\) such that \(Q \langle x,y \rangle=\gamma(x)\);
\item[(iii)] bijective if \(Q\) is both injective and surjective.
\end{enumerate}
\end{definition}

\begin{definition}(Fuzzy composition)\\
If \(Q:\gamma\to U\) and \(G:U\to V\) (where \(V\in I_Z\)) are proper functions, then the proper function \(G \circ Q:\gamma\to V\) is defined by
\[
(G\circ Q)(x,z) =
\begin{cases}
\gamma(x) & \text{if } \ \exists \ y\in Y \text{ such that } \ Q(x,y)=\gamma(x) \text{ and } \ G(y,z)=U(y),\\
0 & \text{otherwise}.
\end{cases}
\]
\end{definition}

\begin{definition}(Fuzzy composition of inverse function)\\
\(G:U\to\gamma\) is called an inverse of \(Q:\gamma\to U\) if \(G\circ Q = I_\gamma\) and \(Q\circ G = I_U\).
\end{definition}

\begin{definition}[Fuzzy homeomorphism][3]\\
A proper function \(Q:(\gamma,\tau)\to (U,\tau_1)\) is said to be
\begin{enumerate}
\item[(i)] fuzzy continuous if \(Q^{-1}(U)\in\tau\) for all \(U\in\tau_1\);
\item[(ii)] fuzzy open if \(Q(\delta)\in\tau_1\) for all \(\delta\in\tau\);
\item[(iii)] fuzzy homeomorphism if \(Q\) is bijective, fuzzy continuous and open.
\end{enumerate}
\end{definition}
\section{Fuzzy Measure Space}
\begin{definition}[Power Set] [13]\\
    The class of the subsets of $X$ is called the power set of $X$, denoted by $\mathcal{P}(X)$
\end{definition}
\begin{definition}[Ring] [13]\\
    The non-empty class $\mathfrak{R}$ is called a ring iff, $\forall \ M, N \in \mathfrak{R}$
    $$
    M \cup N \in \mathfrak{R} \ \text{and} \ M-N\in \mathfrak{R}
    $$
\end{definition}
\begin{definition}[Algebra] [13]\\
    A nonempty class $\mathfrak{R}$ is called an algebra iff, $\forall \ M,N \in \mathfrak{R}$
    $$
    M \cup N \in \mathfrak{R} \ \text{and} \ \overline{M} \in \mathfrak{R}
    $$
\end{definition}
\begin{definition}[Semiring] [13]\\
    A nonempty class $\mathcal{J}$ is called a semiring iff,
    \begin{enumerate}
        \item for all $M, N \in \mathcal{J}$, $M \cap N \in \mathcal{J}$.
        \item for all $M, N \in \mathcal{J}$ satisfying $M \subset N$, there exist  a finite class $\{A_0, A_1, \cdots ,A_n\}$ of sets in $\mathcal{J}$, such that
        $$
        M=A_0, \subset A_1 \subset \cdots A_n = N \\
        \text{and} \ D_i = C_i - C_{i-1} \subset \mathcal{J} \ \text{for all} \ i=1,2, \cdots, n
        $$
    \end{enumerate}
\end{definition}
\begin{definition}[$\sigma$-ring] [13]\\
    A nonempty class $F$ is called a $\sigma$-ring iff,
    \begin{enumerate}
        \item for all $M, N \in F$, $M-N \in F$
        \item for all $M_i \in F \quad i= 1, 2, \cdots$
        $$
        \displaystyle \bigcup_{i=1}^{\infty} M_i \in F
        $$
    \end{enumerate}
\end{definition}
\begin{definition}[$\sigma$-algebra] [13]\\
A $\sigma$ algebra is a $\sigma$-ring that contains $X$
\end{definition}
\begin{definition}[Monotone Class] [13]
    A nonempty class $\mathcal{M}$ is called a monotone class, iff, for every monotone sequence $\{M_n\} \subset \mathcal{M}$, we have $\displaystyle\lim_n M_n \in \mathcal{M}$
\end{definition}
The following theorem establishes that a ring is generated by semiring as can be seen in [13]
\begin{theorem}
    Let $\mathcal{J}$ be a class. There exists a unique ring $\mathfrak{R}_0$, such that it is the smallest ring containing $\mathcal{J}$, such that
    $$
    \mathfrak{R}_0 \supset \mathcal{J}
    $$
    and for ring $\mathfrak{R}$
    $$
    \mathfrak{R} \supset \mathcal{J} \implies \mathfrak{R} \supset \mathfrak{R}_0
    $$
    $\mathfrak{R}_0$ is called the ring generated by $\mathcal{J}$ and is denoted by $\mathfrak{R}(\mathcal{J})$.
\end{theorem}
In the same way, we can define the concepts of $\sigma$-ring and monotone class generated $\mathcal{J}$, and use $F(\mathcal{J})$, $\mathcal{M}(\mathcal{J})$ to denote them, respectfully.
The following theorems are proved in [13]
\begin{theorem}
    Let $\mathcal{J}$ be a semiring. $\mathfrak{R}(\mathcal{J})$ is the class of all finite, disjoint unions of sets in $\mathcal{J}$.
\end{theorem}
\begin{theorem}
    $F(\mathcal{J})=F(\mathfrak{R}(\mathcal{J}))$.
\end{theorem}
\begin{definition} (Measurable space)\\ Let \( X \) be non-empty set. A measure space is defined on \(X\) if the following conditions  are satisfied.
\begin{enumerate}[label=$\square$]
    \item[1]  A $\sigma$-algebra $F$ exist on the nonempty subsets of sets $X$.
    \item[2]  A function $\mu$ is defined on $X$ such that $\mu(P) \geq 0$ and $\mu(\displaystyle\bigcup_{i=1}^{\infty} P_i) = \displaystyle\sum_{i=1}^{\infty} \mu(P_i)$ $\forall$ $P \in F$.
\end{enumerate}
$X$ is called a measurable space, if $X \in F$.
\end{definition}

\begin{definition} [Fuzzy Measure] [13] \\
Let $X$ be a non-empty set, and a collection of classes of subsets of $X$ be denoted as $\tau$. If $\mu$ is defined as a measure such that\\
\[
\mu: \tau \rightarrow [0, \infty].
\] \\
Then $\mu$ is a fuzzy measure on $(X, \tau)$ if the following conditions are satisfied:
\begin{enumerate}
    \item $\mu(\emptyset) = 0$ when $\emptyset \in \tau$.
    \item $\mathbb{P} , \mathbb{Q} \in \tau$ and $\mathbb{P} \subset \mathbb{Q}$ imply $\mu(\mathbb{P}) \leq \mu(\mathbb{Q})$ (monotonicity).
    \item $\{\mathbb{P}_n\} \subset \tau, \ \mathbb{P}_1 \subset \mathbb{P}_2 \subset \cdots$ and \ $\displaystyle\bigcup_{n=1}^{\infty} \mathbb{P}_n \in \tau \quad \text{imply,} \
    \lim\limits_{n} \mu(\mathbb{P}_n) = \mu\left(\displaystyle\bigcup_{n=1}^{\infty} \mathbb{P}_n\right) \quad \text{(Continuity from below).}$
    \item $\{\mathbb{P}_1\} \subset \tau,\ \mathbb{P}_1 \supseteq \mathbb{P}_2 \supseteq \cdots$, $\mu(\mathbb{P}_1) < \infty$ and \
          $\displaystyle\bigcap_{n=1}^{\infty} \mathbb{P}_n \in \tau \quad \text{imply}, \
          \lim\limits_{n} \mu(\mathbb{P}) = \mu\left(\displaystyle\bigcap_{n=1}^{\infty} \mathbb{P}_n\right) \quad \text{(Continity from above).}$
\end{enumerate}
Hence $\mu$ is a lower or upper semi-continuous fuzzy measure on $(X, \tau)$. If the condition $(1)$ $(2)$ and $(3)$ or $(1)$ $(2)$ and $(4)$ above are satisfied respectively.
\end{definition}

$\mu$ is regular iff $\mu(X) = 1$.

Hence $(X, F, \mu)$ is a fuzzy measure space if $\mu$ is a fuzzy measure on a measurable space $(X, F)$.

\begin{Example}
    Let $\mu$ be fuzzy measurabe on $(X, F)$, for all $\mathbb{P} \in F$, if
\[
\mu(\mathbb{P}) =
\begin{cases}
1, & a_0 \in \mathbb{P} \\
0, & a_0 \notin \mathbb{P} \quad \quad ,
\end{cases}
\]
such that there exists a set point $a_0$ in $X$, then $\mu$ is a probability measure and hence a regular fuzzy measure.
\end{Example}

\begin{Example}

Let $X = \mathbb{N}$, define $\tau = \mathcal{P}(X), \ \text{where} \ \mathcal{P}(X) = \text{power set of X},$ for all $\mathbb{P} \in F$, if
\[
\mu(\mathbb{P}) = \left( \frac{|\mathbb{P}|}{n} \right)^2 \quad \quad ,
\]
where $|\mathbb{P}|$ is the cardinality of $\mathbb{P}$, then $\mu$ is a regular fuzzy measure.
\end{Example}

\iffalse
\begin{definition}[Autocontinuity] [13]
    A fuzzy measure $\mu$ is said to auto continuous iff the map: $\mu : \ F \to [-\infty, \infty]$ satisfies the following conditions
    \[
    \lim_n \mu (P_n \cup P) = \mu (P)
    \]
    \[
    \bigg( or \ \lim_n \mu (P_n - P) = \mu (P)\bigg)
    \]
    whenever $P_n, P \in F$, and $P_n \cap P = \emptyset \ (or \ P \subset P_n, \quad respectively) $, $n = 1, 2, ...$ and $\displaystyle\lim_n \mu (P_n) = 0$.

\end{definition}
\fi

\subsection{Null-additivity}

\begin{definition}[Null-additivity][13]\\
A fuzzy measure \(\mu\) is said to be \textbf{null-additive} if and only if
\[
\mu(P \cup Q) = \mu(P)
\]
whenever \(P, Q \in F\), \(P \cap Q = \emptyset\) and \(\mu(Q) = 0\).
\end{definition}

The following theorems are proved in [13]
\begin{theorem}
If for any nonempty set \(Q \in F\), \(\mu(Q) \neq 0\), then \(\mu\) is null-additive.
\end{theorem}

\begin{theorem}
Let \(\mu\) be a null-additive fuzzy measure and \(P \in F\). Then we have
\[
\lim_n \mu(P \cup Q_n) = \mu(P)
\]
for any decreasing set sequence \(\{Q_n\} \subset F\) for which \(\displaystyle\lim_n \mu(Q_n) = 0\) and there exists at least one positive integer \(n_0\) such that \(\mu(P \cup Q_{n_0}) < \infty\) as \(\mu(P) < \infty\).
\end{theorem}

\begin{theorem}
Let \(\mu\) be a null-additive fuzzy measure and \(P \in F\). Then we have
\[
\lim_n \mu(P - Q_n) = \mu(P)
\]
for any decreasing set sequence \(\{Q_n\} \subset F\) for which \(\displaystyle\lim_n \mu(Q_n) = 0\).
\end{theorem}

\subsection{Auto-continuity}

\begin{definition}[Auto-continuity][13]\\
A fuzzy measure \(\mu\) is said to be \textbf{auto-continuous} iff the map \(\mu: F \to [-\infty, \infty]\) satisfies the following conditions.
\[
\lim_n \mu(P \cup Q_n) = \mu(P)
\]
or
\[
\lim_n \mu(P - Q_n) = \mu(P)
\]
whenever \(P \in F\), \(Q_n \in F\), \(P \cap Q_n = \emptyset\) (or \(Q_n \subset P\), respectively), \(n = 1, 2, \dots\) and \(\displaystyle\lim_n \mu(Q_n) = 0\).
\end{definition}

The following theorems are proved in [13]
\begin{theorem}
Let \(\mu: F \to [-\infty, \infty]\) be an extended real-valued set function. If there exists \(\epsilon > 0\) such that \(|\mu(P)| \ge \epsilon\) for any \(P \in F\), \(P \neq \emptyset\), then \(\mu\) is auto-continuous.
\end{theorem}

\begin{theorem}
Let \(\mu: F \to [0, \infty]\) be nondecreasing. \(\mu\) is auto-continuous if and only if
\[
\lim_n \mu(P \Delta Q_n) = \mu(P)
\]
wherever \(P \in F\), $\{Q_n\} \subset F$ and \(\lim_n \mu(Q_n) = 0\). \(P \Delta Q_n = P \cup Q\) whenever \(P \cap Q = \emptyset\).
\end{theorem}

\subsection{Uniform auto-continuity}

\begin{definition}[Uniform auto-continuity][13]\\
A fuzzy measure \(\mu\) is said to be \textbf{uniformly auto-continuous} iff for any $\epsilon > 0$, there exists $\delta=\delta_\epsilon > 0$, such that the map \(\mu: F \to [-\infty, \infty]\) satisfies the following conditions:
\[
\mu(P) - \epsilon \le \mu(P \cup Q) \le \mu(P) + \epsilon
\]
or
\[
\mu(P) - \epsilon \le \mu(P - Q) \le \mu(P) + \epsilon
\]
whenever \(P, Q \in F\), \(P \cap Q = \emptyset\) (or \ \(Q \subset P\), respectively) and \(|\mu(Q)| \le \delta\). A fuzzy measure \(\mu\) is uniformly auto-continuous iff it is both uniformly auto-continuous from above and uniformly auto-continuous from below.
\end{definition}

\begin{theorem}[NEW]
Every fuzzy measure is complete.
\end{theorem}

\begin{proof}
For any \(P \in F\), every sequence \(\{Q_n\} \subset F\) with \(\displaystyle\lim_n \mu(Q_n) = 0\), we have \(Q_n - P \in F\) and \(\mu(Q_n - P) \le \mu(Q_n)\). Since every bounded infinite sequence has a monotone subsequence, there exists a subsequence \(\{Q_{n_k}\} \) of \(\{Q_n\}\) where \(k=1,2,\dots\) such that
\[
\lim_k \mu\left(\bigcup_{k=1}^{\infty} Q_{n_k}\right) = 0.
\]
Now we have
\[
0 \le \mu(Q_{n_k} - P) \le \mu(Q_{n_k}) = 0.
\]
However
\[
\mu(P - Q_{n_k}) = \mu(P - (Q_{n_k} - P)) \quad \text{and} \quad Q \cap P \subset P,
\]
where
\[
0 \le \mu(Q_{n_k} \cap P) \le \mu(Q_{n_k}) = 0.
\]
Hence by uniform auto-continuity, $\forall \ \epsilon > 0$ there exists \(\delta = \delta_\epsilon > 0\) such that \(|\mu(Q_{n_k})| \le \delta\).
Since
\[
\mu(P \cap Q_{n_k}) \le \mu(Q_{n_k}) \le \delta.
\]
It follows that
\[
\mu(P) = \mu((P - Q_{n_k}) \cup (P \cap Q_{n_k})) \le \mu(P - Q_{n_k}) + \epsilon.
\]
Also since
\[
\mu(Q_{n_k} - P) \le \mu(Q_{n_k}) \le \delta,
\]
we have
\[
\mu(P) \ge \mu(P - Q_{n_k}) = \mu[(P \Delta Q_{n_k}) - (Q_{n_k} - P)] \ge \mu(P \Delta Q_{n_k}) - \epsilon.
\]
Hence
\[
\mu(P \Delta Q_{n}) - \epsilon \le \mu(P) \le \mu(P - Q_{n}) + \epsilon.
\]
Thus \(Q_n\) is fuzzy measurable. If \(Q_n\) is a fuzzy measurable set sequence with \(\mu(Q_n) = 0\), we have \(Y \in Q_n\), where \(\mu(Y) = 0\). Since \(\mu(Y) \le \mu(Q_n)\), \(Y\) is fuzzy measurable, and this completes the proof.
\end{proof}

\section{Measurable functions on fuzzy measurable space}
\begin{definition} [Borel function] [13] \\
Suppose we defined $\mathbb{R} = (-\infty, \infty)$ and define $\mathbb{R}^n = \mathbb{R} \times \mathbb{R} \times \mathbb{R} \times \cdots \times \mathbb{R}$ an $n-$dimensional product spaces. Now, we denote
\[
\mathcal{J}^{(n)} = \left\{ \prod_{i=1}^{n} ({\mathfrak{a} _i, \mathbf{b}_i}) : \quad -\infty < \mathfrak{a}_i \leq \mathfrak{b}_i < \infty, \ i=1,2,...,n \right\},
\]
$\sigma$-algebra $B^{(n)} = F(\mathcal{J}^{(n)})$ is a Borel field on $\mathbb{R}^n$. The sets in $B^{(n)}$ are said to be $n$-dimensional Borel sets.
Note: A function defined as
\[
f: \mathbb{R}^n \rightarrow \mathbb{R}
\]
is a Borel function if and only if it defines a measurable function on the measurable space $(\mathbb{R}^n, B^{(n)})$.
\end{definition}

\begin{definition} [Measurable function] [13] \\
Given a measurable space $(X, F)$ and a fuzzy measure defined as $\mu: F \rightarrow [0, \infty)$, $\exists$ a Borel field $B$ on $(-\infty, \infty)$ such that:
\[
f: X \rightarrow (-\infty, \infty) \text{ on } X \text{ is measurable,}
\]
then,
\[
f^{-1}(W) = \{a \mid f(a) \in W \} \in F.
\]
Note: If $\{a \mid f(a) = c\} \in F$ for all $c \in (-\infty, \infty)$, then $f$ is called a measurable function
\end{definition}

\section{Fuzzy Integrals:}
\begin{definition} [$\lambda$-cut set] [13] \\
Given a measurable space $(X, F)$, a finite non-negative class of all measurable functions defined on measurable space is denoted as $\mathcal{F}$.

\begin{itemize}
    \item $\lambda$-cut set: For every $f \in \mathcal{F}$, define $\mathcal{F}_{\lambda} = \{ a \mid f(a) \geq \lambda \}$.
    \item Strict $\lambda$-cut set: $\forall \ f \in \mathcal{F}$, $\mathcal{F}_{\lambda+} = \{ a \mid f(a) > \lambda \}$, $\forall \ a \in X$ and $\lambda \in [0,\infty]$
\end{itemize}
\end{definition}

\begin{definition} [Fuzzy Integral] [13] \\
Suppose $P \in F$ and $f \in \mathcal{F}$. Then we define the fuzzy integral of $f$ over $P$ relative to the fuzzy measure $\mu$ by
\[
\int_{P} f \, d\mu = \sup_{\lambda \in [0, \infty]} [\lambda \cap \mu(P \cap \mathcal{F}_{\lambda})]
\]
When $P$ = $X$, the fuzzy integral is denoted as $\displaystyle \int f \, d\mu$.
\end{definition}
The following lemma gives the properties of fuzzy integral as can be seen in [13]
\begin{lemma}
Both $\mathcal{F}_{\lambda}$ and $\mathcal{F}_{\lambda+}$ are non-increasing with respect to $\lambda$ and $\mathcal{F}_{\lambda+} > \mathcal{F}_{\delta}$ where $\lambda < \delta$.
\end{lemma}
$\textbf{Rules of fuzzy Integral}$ [13]
\begin{enumerate}
    \item $\mu(P) = 0$, then $\displaystyle \int_P f \, d\mu = 0$ for any $f \in \mathcal{F}$.
    \item $\displaystyle \int f \, d\mu = 0$, then $\mu(P \cap \{a \mid f(a) > 0\}) = 0$.
    \item $f_1 \leq f_2$, then $\displaystyle \int_P f_1 \, d\mu \leq \displaystyle \int_P f_2 \, d\mu$.
    \item $\displaystyle \int_P f \, d\mu = \displaystyle \int f \cdot \chi_P \, d\mu$, where $\chi_P$ is the characteristic function of $P$.
    \item $\displaystyle \int_P m \, d\mu = m \cdot \mu(P)$ for any $m \in [0, \infty]$.
    \item $\displaystyle \int_P (f + m) \, d\mu \leq \displaystyle \int_P f_1 \, d\mu + \displaystyle \int_P m \, d\mu$.
    \item If $Q \subset P$, then $\displaystyle \int_Q f_1 \, d\mu \geq \displaystyle \int_Q f_2 \, d\mu$.
    \item $\displaystyle \int_P (f_1 \vee f_2) \, d\mu = \displaystyle \int_P f_1 \, d\mu \vee \displaystyle \int_P f_2 \, d\mu$.
    \item $\displaystyle \int_P (f_1 \land f_2) \, d\mu = \displaystyle \int_P f_1 \, d\mu \land \displaystyle \int_P f_2 \, d\mu$.
    \item $\displaystyle \oint_{P_\vee Q} f \, d\mu \geq \displaystyle \oint_{P} f \, d\mu \lor \displaystyle \oint_{Q} f \, d\mu$
    \item $\displaystyle \oint_{P_\land Q} f \, d\mu \leq \displaystyle \oint_{P} f \, d\mu \land \displaystyle \oint_{Q} f \, d\mu$
\end{enumerate}

\section{Fuzzy Integrals defined on fuzzy measures.}
The following theorems establishes a transform from a fuzzy integral $\displaystyle\int_{P} f d\mu = I_f$ defined on a fuzzy measure space to another fuzzy integral $I_g$ defined on the Lebesgue measure space and, defines a fuzzy measure using the fuzzy integral of a given measurable function as can be seen in [13]
\begin{theorem}[][13]
    For any $P \in F$
    $$
    \int_{P}f d\mu = \int \mu(P \cap \mathcal{F}_{\lambda}) d\nu
    $$
    where $\mathcal{F}_{\lambda} = \{a|f(a) \geq \lambda \}$ and $\nu$ is the Lebesgue measure.
\end{theorem}

\begin{theorem}[][13]
Given the fuzzy measure space $(X, F, \mu)$ and a measurable function $f \in \mathcal{F}$. $H$ is denoted as set function such that:
\[
H(P) = \int_{P} f \, d\mu
\]
$\forall \ P \in F$ is a lower semicontinuous fuzzy measure on $(X,F)$. Hence if $\mu$ is finite, then $H$ is a finite fuzzy measure on $(X,F)$
 \end{theorem}

\section{Fuzzy Topological Group and Fuzzy Lie Group}

\begin{definition}[Fuzzy $T_1$][9]\\
    A fuzzy topological space is called a fuzzy $T_1$ space if every fuzzy point is a closed fuzzy set.
\end{definition}

\begin{definition}[][9]\\
    Let $E_1, E_2$ be two fuzzy topological vector spaces and let $\phi$ be a mapping from $E$ into $E_2$. Let $0(t)$ denote any function of a real variable $t$ such that
    $$
    \lim_{t \to 0} 0 (t)/t = 0
    $$
    The mapping $\phi$ is said to be tangent to $0$ if given a neighborhood $W$ of $0_{\delta}, 0 < \delta \leq 1$, in $E_2$ there exists a neighbourhood $V$ of $O_{\gamma},$ for every $\gamma$, $0 < \gamma < \delta$ in $E_1$ such that\\
    $$
    \phi [tV] \subset 0(t)W
    $$
    for some function  $0(t)$
 \end{definition}

\begin{definition}[Fuzzy differentiability][9] \\
Let $E_{1}$ and $E_{2}$ be two fuzzy topological vector spaces, each endowed with a $T_1$ fuzzy topology. Let $f: E_{1} \to E_{2}$ be a fuzzy continuous mapping. Then $f$ is said to be \emph{fuzzy differentiable} at a point $x \in E_{1}$ if there exists a linear fuzzy continuous mapping $U$ of $E_{1}$ into $E_{2}$ such that
\[
f(x+y)=f(x)+U(y)+\phi(y), \quad y \in E_{1},
\]
where $\phi$ is tangent to $0$. The mapping $U$ is called the \emph{fuzzy derivative} of $f$ at $x$. The fuzzy derivative of $f$ at $x$ is denoted by $f'(x)$; it is an element of $L(E_{1},E_{2})$, the set of all linear fuzzy continuous mappings of $E_{1}$ into $E_{2}$. The mapping $f$ is \emph{fuzzy differentiable} if it is fuzzy differentiable at every point of $E_{1}$.
\end{definition}

\begin{definition}[Fuzzy diffeomorphism][9]\\
Let $E_{1}, E_{2}$ be fuzzy topological vector spaces. A bijection $f$ of $E_{1}$ onto $E_{2}$ is said to be a \emph{fuzzy diffeomorphism of class $C^1$} if and only if $f$ and its inverse $f^{-1}$ are fuzzy differentiable, and $f'$ and $(f^{-1})'$ are fuzzy continuous.
\end{definition}

\begin{definition}[Fuzzy atlas][3]\\
Let $X$ be a set. A \emph{fuzzy atlas $\mathcal{A}$ of class $C^1$} fuzzy atlas on $X$ is a collection of pairs $(A_j, \phi_j)$, $j \in J$, which satisfies the following conditions:
\begin{enumerate}
\item Each $A_j$ is a fuzzy set in $X$ and $\sup\{\cup{A_j}(x)\}=1$ for all $x \in X$.
\item Each $\phi_j$ is a bijection defined on the support of $A_j$, $\{x\in X: \cup{A_j}(x)>0\}$, which maps $A_j$ onto an open fuzzy set $\phi_j[A_j]$ in some fuzzy topological vector space $E_j$, and for each $i$ in the index set, $\phi_j[A_j \cap A_i]$ is an open fuzzy set in $E_j$.
\item The mapping $\phi_i \circ \phi_j^{-1}$, which maps $\phi_j[A_j \cap A_i]$ onto $\phi_i[A_j \cap A_i]$, is a $C^1$ fuzzy diffeomorphism for each pair of indices $j,i$.
\end{enumerate}
Each pair $(A_j,\phi_j)$ is called a \emph{fuzzy chart} of the fuzzy atlas. If a point $x \in X$ lies in the support of $A_j$, then $(A_j,\phi_j)$ is said to be a fuzzy chart at $x$.
\end{definition}
Let $(X,\tau)$ be a fuzzy topological space. Suppose there exist an open fuzzy set $A$ in $X$ and a fuzzy continuous bijective mapping $\phi$ defined on the support of $A$ and mapping $A$ onto an open fuzzy set $V$ in some fuzzy topological vector space $E$. Then $(A,\phi)$ is said to be \emph{compatible} with the $C^1$ atlas $\{(A_j,\phi_j)\}$ if each mapping $\phi_j \circ \phi^{-1}$ of $\phi[A \cap A_j]$ onto $\phi_j[A \cap A_j]$ is a fuzzy diffeomorphism of class $C^1$.\\

Two $C^1$ fuzzy atlases are compatible if each fuzzy chart of one atlas is compatible with each fuzzy chart of the other atlas. Also, the relation of compatibility between $C^1$ fuzzy atlases is an equivalence relation.

\begin{definition}[$C^1$ fuzzy manifold][3]\\
    An equivalence class of $C^1$ fuzzy atlases on $X$ is said to define a \emph{$C^1$ fuzzy manifold} on $X$.
\end{definition}

\begin{definition}[Fuzzy differentiable mapping between manifolds][9]\\
Let $X, Y$ be fuzzy manifolds and let $f: X \to Y$ be a mapping. Then $f$ is said to be \emph{fuzzy differentiable} at a point $x \in X$ if there exists a fuzzy chart $(U,\phi)$ at $x \in X$ and a fuzzy chart $(V,\psi)$ at $f(x) \in Y$ such that the mapping $\psi \circ f \circ \phi^{-1}$, which maps $\phi[U \cap f^{-1}[V]]$ into $\psi[V]$, is fuzzy differentiable at $\phi(x)$.

The mapping $f$ is \emph{fuzzy differentiable} if it is fuzzy differentiable at every point of $X$; it is a $C^1$ fuzzy diffeomorphism if $\psi \circ f \circ \phi^{-1}$ is a $C^1$ fuzzy diffeomorphism.

The concept of a fuzzy Lie group depends on the basic concepts in fuzzy topology, $C^1$-fuzzy manifolds, and fuzzy differentiable functions between two $C^1$-fuzzy manifolds, all of which have been discussed above.
\end{definition}

\begin{definition} [Fuzzy Group] [3]\\

    $G$ is called a fuzzy group if a mapping
\[
\mathfrak{Y} : G \to [0, 1] \quad \text{satisfies the following conditions:}
\]
\[
\mathfrak{Y}(x y) \geq \min \{\mathfrak{Y}(x), \mathfrak{Y}(y)\} \quad \forall \ x,y \in G.
\]
\[
\mathfrak{Y}(x^{-1}) \geq \mathfrak{Y}(x).
\]
\end{definition}

\begin{definition} [Fuzzy Topological Group] [3] \\
Given a fuzzy group $G$. A fuzzy topological group is defined as the pair $(G, \tau)$ where $\tau$ is a fuzzy topology defined on $G$. If the maps:
\[
q : G \times G \to G \ | \ (h_1, h_2) \mapsto h_1 \cdot h_2
\]
and
\[
i : G \to G \ | \ h \mapsto h^{-1}
\]
are fuzzy continuous.
\end{definition}

Then the pair $(G, \tau)$ is called a fuzzy topological Group.
\begin{definition} [Fuzzy Lie Group] [3] \\

    A fuzzy Lie group $G$, is a $C^1$- fuzzy manifold $G$ such that the mapping
    \[
    p: G  \times  G  \to G \ | \ (k_1, k_2) \mapsto k_1 \cdot k_2
    \]
    \[
    i: G \to G \ | \ k \mapsto k^{-1}
    \]
    are fuzzy differentiable
\end{definition}

\begin{definition} (Locally Compact Fuzzy Lie Group)\\
Given a fuzzy topological group $G$. If $G$ possesses the properties of locally compact and Hausdorff. Then $G$ is said to be a locally compact fuzzy Lie group.
\end{definition}

\section{Fuzzy Haar Measure}

\begin{definition} [Fuzzy Haar Measure] [NEW]

A fuzzy measure $\mu$ defined on a locally compact fuzzy Lie group $\mathfrak{G}_f$ is called a fuzzy Haar measure denoted by $\mu_f$ if $\mu$ satisfies the following conditions
\begin{itemize}
    \item[1]  $\mu(\emptyset) = 0 \quad where \quad \emptyset \in \tau.$
    \item[2]  For any $P, \ Q \subset \mathfrak{G}_f$ such that $P \subset Q$, Then $\mu(P) \leq \mu(Q)$ \quad (Monotonicity)
    \item[3]  $\lim_{n \to \infty} \mu(P_n) = \mu(P)$ for any monotone sequence $P_n \to P$ \quad (Continuity)
    \item[4]  $\mu(gP) = \mu(P)$ for all $g \in \mathfrak{G}_f$ \quad (Left invariance)
\end{itemize}
\end{definition}

\begin{Example}
    Given a function $f$ defined on $\mathfrak{G}_f$, it is non-negative.\\
    Let
\[
\mathfrak{G}_f = (-\infty, \infty) \quad and
\]
\[
\mu(P) = \displaystyle\sup_{a \in P} f(a)
\]
$\forall \quad P \in \mathcal{P} (\mathfrak{G}_f) $.

We show that $\mu$ satisfy the set conditions of a fuzzy Haar measure. To see this, we shall show that \\
    \begin{enumerate}
    \item \[\mu(P) = \sup_{a \in P} f(a).\] Let $P = \emptyset $, $$\mu(\emptyset) = \sup_{x\in \emptyset} f(x).$$ But $x\notin \emptyset$ $$\mu(\emptyset)=\sup_{x \in \emptyset} f(x) = 0 $$
    \item $\forall \ P \in \mathcal{P} (\mathfrak{G}_f), \ Q \in \mathcal{P} (\mathfrak{G}_f) \quad \text{and} \quad P \subset Q,$ \\

    We have
          \[
          \mu(P) = \sup_{a \in P} f(a)
          \]
          \[
          \leq \sup_{b \in Q} f(b)
          \]
          \[
          = \mu(Q)
          \]
          Hence $\mu(P) \leq \mu(Q)$.

    \item
          $\displaystyle\lim_{n} \mu(P_n) = \mu(P)$ \\ We have
    $$\begin{array}{rcl}
    \displaystyle\lim_{n} \mu(P_n - P) &=& 0, \\
    \displaystyle\lim_{n} \mu(P_n) &=& \displaystyle\lim_{n} \mu(P \lor (P_n - P))\\
    &=& \mu(P).
    \end{array} $$
    Thus \\
    $$ \begin{array}{rcl}
    \displaystyle\lim_{n} \displaystyle\sup_{a_n \in P_n} f(a_n) &=& \displaystyle\lim_{n} \displaystyle\sup_{a_n \in P_n} (f(a) \lor (f(a_n) - f(a)) ) \\
    &=& \displaystyle\sup_{a \in P} f(a).
    \end{array}$$

    Hence $\mu$ is continuous from below.

    However,

$$ \begin{array}{rcl}
            \displaystyle\lim_{n} \mu(P_n - P) &=& 0 \\
            \displaystyle\lim_{n} \mu(P_n - P) &=& \displaystyle\lim_{n} \mu(P \land (P_n -P) = \mu(P) \\
            \displaystyle\lim_{n} \displaystyle\sup_{a_n \in P_n} f(a_n) &=& \displaystyle\sup_{a \in P} f(a_n) \\
            \displaystyle\lim_{n } \displaystyle\sup_{a_n \in P_n} (f(a_n) - f(a)) &=& 0 \\
            \displaystyle\lim_{n} \displaystyle\sup_{a_n \in P_n} (f(a_n)) &=& \displaystyle\lim_{n} \displaystyle\sup_{a_n \in P_n} ((f(a) \land (f(a_n) - f(a))) \\
            &\neq& \displaystyle\sup_{x \in A} f(a)
          \end{array} $$

          This shows that the continuity from above is not satisfied.

    \item
    $$ \begin{array}{rcl}
    \mu (gP - P) &=& 0 \\
    \displaystyle\sup_{a  \in P} f(ga) &=& \displaystyle\sup_{a \in P} (f((g_1,g_2,g_3)(a_1, a_2, a_3))) \\
    &=& \displaystyle\sup_{a \in P} f (g_1a_1,g_2a_2,g_3a_3) \\
    &=& \displaystyle\sup_{a_1 \in P} f(g_1a_1)\displaystyle\sup_{a_2 \in P} f(g_2a_2)\displaystyle\sup_{a_3 \in P} f(g_3a_3) \\
    &=& \displaystyle\sup_{a_1 \in P} f(a_1) \displaystyle\sup_{a_2 \in P} f(a_2) \displaystyle\sup_{a_3 \in P} f(a_3) \\
    &=& \displaystyle\sup_{a \in P} f(a).
    \end{array}$$
\end{enumerate}
Hence $\mu_f$ is a fuzzy Haar measure.
\end{Example}

\begin{theorem}
A complete fuzzy measure is a fuzzy Haar measure if and only if it is translation invariant over a fuzzy Lie group \(\mathfrak{G}_f\).
\end{theorem}

\begin{proof}
Suppose \(\mu\) is a complete fuzzy measure. If \(\mu\) is a fuzzy Haar measure denoted as \(\mu_f\), then for any \(T \in \mathfrak{G}_f\) and \(B \in F\),
\[
\mu_f(T * b) = \mu_f(T), \quad \forall \ b \in B.
\]
But there exists \(D_i \subset F\) such that
\[
T \subset \displaystyle\bigcup_i D_i, \quad i = 1, 2, \dots
\]
Also for any \(D_i\), \(i=1,2,\dots\), there exist \(T * b \subset \displaystyle\bigcup_i D_i\), \(i=1,2,\dots\), and we have $$T \subset \bigcup_i (D_i - b)$$. Since \(\limsup \mu(D_i - b) = m\) where \(m \in [0,\infty)\). However \(\lim \mu(D_i) = \lim \mu(D_i - b)\). Hence \(\limsup \mu(D_i) = m\). It follows that \(D_i\), \(i=1,2,\dots\) covers \(T * b\) and we have $$\mu_f(T * b) = \mu_f(T)$$.

Conversely, if \(\mu_f\) is translation invariant, there exists a covering \(D_i \subset F\) such that for all \(T \subset \mathfrak{G}_f\), \(T \subset \displaystyle\bigcup_i D_i\), \(i=1,2,\dots\). However for any $b \in B$, $T * b \subset \displaystyle\bigcup_i D_i.$ Hence $$\mu_f(T * b) = \mu_f(T).$$
\end{proof}

\begin{definition} [Unimodular]
    A fuzzy Lie group $\mathfrak{G}_f$ is said to be unimodular iff the left and right Haar measures coincide.
\end{definition}

\section{Fuzzy Haar Integral}
\begin{definition} [Fuzzy Haar Integral]

Given a locally compact fuzzy Lie group $\mathfrak{G}_f$. If $P \in F$ and $f \in \mathcal{F} $. Then $I$ is said to be a fuzzy Haar integral of $f$ over $\mathfrak{G_f}$ with respect to the fuzzy Haar measure $\mu_f$ denoted as $I_{\mathfrak{G_f}} (f)$  = $\int_{\mathfrak{G}_f} f \, d\mu_f$ if the following conditions are satisfied:
\begin{itemize}
    \item[1]   \[
\int_{P} f(gP) \, d\mu_f = \sup_{\lambda \in [0,\infty)} [f(g) \lor [\lambda \land \mu_f(P \cap \mathcal{F}_{\lambda}) = \int_{P} f(P)
\]
\item[2]
\[
\int_{P} f(P) \, d\mu_f \geq 0
\]
\end{itemize}
$ \forall \quad g \in \mathfrak{G_f}$.
\end{definition}

\begin{Example}

 If $\mathfrak{G}_f = \{x,y,z\}$ and $F = \mathcal{P}(\mathfrak{G}_f) $, the measure $\mu_f$ as defined as
\[
\mu_f(P) =
\begin{cases}
|P| & \text{if } P \neq \{x,y\} \\
1 & \text{if } P = \{x,y\}
\end{cases}
\]
$\forall \quad P \in F$. Let
\[
f(a) =
\begin{cases}
\frac{1}{2} & \text{if } a = x \\
1 & \text{if } a = y \\
2 & \text{if } a = z \quad \quad \forall \ a \in P
\end{cases}
\]
Show that $\mu_f$ is a fuzzy Haar Measure.
\end{Example}

\textbf{Solution}:
\begin{enumerate}
    \item  From
\[
\int_{P} f(gP) \, d\mu_f = \sup_{\lambda \in [0,\infty]} [f(g) \lor [\lambda \wedge \mu_f (P \cap \mathcal{F}_{\lambda})] = \int_{P} f(P) \, d\mu_f \ ,
\]
If $f(g)=\frac{|g|}{3}$ and $g = x$.\\ \\
$f(1)=\frac{1}{3}$. Hence
\[
\int_{P} f(gP) \, d\mu_f = \sup_{\lambda \in [\frac{1}{2},2]} [f(g) \lor [[\frac{1}{2} \land \mu_f(\{x\})] \lor [1 \land \mu_f (\{x,y\})] \lor [2 \land \mu_f (\{G_f\})] ]]
\]
\[
= \sup \ [\frac{1}{3} \lor [\frac{1}{2} \lor1 \lor 2]]
\]
\[
=\sup \ [\frac{1}{3} \lor 2]
\]
\[
= \frac{1}{3}
\]
\item  Clearly
\[
\int_{P} f(P) \, d\mu_f > 0. \
\]
\end{enumerate}

\begin{theorem}

 Given a locally compact fuzzy Lie group $G_f$, such that \( P, Q\in  F \), and \( f \in \mathcal{F} \). Let \( C_c(G_f) \) be continuous fuzzy function on \( G_f \) with compact support and \( g \in C_c(G_f) \). Then the fuzzy Haar Integral
\[
\int_{P} f \, d\mu_f
\]
is unique up to a multiplicative constant. That is,
\[
\int_{P} f \, d\mu_f = \gamma \int_{Q} f \, d\mu_f
\]
$\forall \quad \gamma > 0$.
\end{theorem}

\textbf{Proof}

Let \( f_n \in F \). We consider
\[
(f(P^{-1})) = f(P).
\]
Also replace \( f_n \) by the function
\[
P^{-1} \to f_n(P)f_n(P^{-1})
\]

If the integrals \(\int_{P}\) and \(\int_{Q}\) are left invariant fuzzy Haar integrals on \(G_f\), we take
\[
\gamma_n = \left(\int_{P}f_n\right)\left(\int_{Q}f_n\right)^{-1}
\]

Clearly \(\gamma_{n} > 0\). for all n, we may assume that either \(\gamma_{n} \to \gamma\) or \(\gamma_{n}^{-1} \to \gamma.\)\\
for some \(\gamma > 0\)
Applying fubini's theorem, we have \\
\[
\int_P \otimes \int_Q  \quad on \quad G_f \quad \times \quad G_f \quad \hbox{together}. \
we \ have \ for \ any \ g \in C_c(G_f) \ and \ for \ every \ function \ f_n
\]
$$\begin{array}{rcl}
      \displaystyle\int_Q \displaystyle\int_P g(P)f_n(Q) \, dP \, dQ  &=& \sup_{\lambda \in [0, \infty]} [[\lambda \land \mu_f (P \cap \mathcal{F_{\lambda} })] \land [\lambda \land \mu_f (Q \cap \mathcal{F_{\lambda}})]] \\ \\
      &=& \sup_{\lambda \in [0, \infty]} [[\lambda \land \mu_f (Q \cap \mathcal{F_{\lambda} })] \land [\lambda \land \mu_f (P \cap \mathcal{F_{\lambda}})]] \\ \\
      &=& \sup_{\lambda \in [0, \infty]} [[f_{n}(P)^{-1} \land [ \lambda \land \mu_f (Q \cap \mathcal{F_{\lambda} })] \land [\lambda \land \mu_f (P \cap \mathcal{F_{\lambda}})]] \\ \\
      &=& \sup_{\lambda \in [0, \infty]} [[\lambda \land \mu_f (P \cap \mathcal{F_{\lambda} })] \land [f(P)^{-1}\lambda \land \mu_f (P \cap \mathcal{F_{\lambda}})]] \\ \\
      &=& \sup_{\lambda \in [0, \infty]} [[g(Q) \land [\lambda \land \mu_f (P \cap \mathcal{F_{\lambda} })] \land [\lambda \land \mu_f (P \cap \mathcal{F_{\lambda}})]] \\ \\
      &=& \sup_{\lambda \in [0, \infty]} [g(Q) \wedge [\lambda \wedge \mu_f(P \wedge \mathcal{F}{\lambda})] \wedge [\lambda \wedge \mu_f(P \wedge \mathcal{F}{\lambda})]] \\ \\
      &=& \displaystyle\int_{Q} \displaystyle\int_{P} g(QP) f_n(P) \, dP \, dQ
\end{array}$$

For all \(\epsilon > 0\), \(\exists\) a Compact neighborhood \(W\), such that \(\forall \, P \in W\)

\[
\int_{Q} |g(QP) - g(Q)| \, dQ < \epsilon
\]

However, the support of \(f_n\) is continued in \(W\). We can establish

$$\begin{array}{rcl}
\left|\displaystyle\int_{P} g(P) \, dP - \gamma \displaystyle\int_{Q} f(Q) \, dQ\right|
&=&\lim_n ( \displaystyle\int_{Q} (f_n)^{-1} \left | \displaystyle\int_{P} \displaystyle\int_{Q} g(P) f_n(Q) - f_n(P) g(Q) \, dQ \, dP \right| \\ \\
&=& \lim_n ( \displaystyle\int_{Q} (f_n)^{-1} \left| \displaystyle\int_{P} \displaystyle\int_{Q} (g(QP)) - g(Q) f_n(P) \, dQ \, dP \right| \\ \\
&\leq& \lim_n \sup \displaystyle\int_{Q} (f_n)^{-1} \displaystyle\int_{P} \epsilon f_n \\ \\
&=& \epsilon \gamma
\end{array}$$
Since \(\epsilon\) and \(g\) were arbitrary we say
\[
\int_{P} f \ {d\mu_f} = \gamma \int_{Q} f \ {d\mu_f}
\]
\\ \\

\begin{theorem}
    A locally compact fuzzy Lie group \(G_f\) is said to be unimodular if and only if $\gamma \ \equiv \ 1$
\end{theorem}
\[
Proof
\]
If $\gamma \ \equiv \ 1$, then \(G_f\) is unimodular, since \\
\[
\int_{P} f(ag) = \int_{P} f(a) = \gamma \int_{Q} f(a)
\]
by definition 7.11 and theorem 9.3 \\
Conversely, if \(G_f\) is unimodular. Let $I_{G_f}$ be a fuzzy Haar integral on \(G_f\). Let $Q \in F$. Then $I_{G_f} (f) \in (0,\infty)$ , for any $f \in \mathcal{F}$


\begin{thebibliography}{99}

\bibitem{akram2018} Akram, M. (2018) \textit{Fuzzy Lie Algebras}. Springer Verlag, Berlin. \url{https://doi.org/10.1007/978-981-13-3221-0}

\bibitem{alfsen1963} Alfsen, E.M. (1963), A simplified constructive proof of the existence and uniqueness of Haar measure,
\textit{Mathematica Scandinavica} 12, 106--116.

\bibitem{egwe2024} Egwe, M.E. and Sangodele, S.S. (2024) ``Spherical Functions on Fuzzy Lie Group, \textit{Advances in Pure Mathematics}, 14, 185--195. \url{https://doi.org/10.4236/apm.2024.144010}

\bibitem{haar1933} Haar, A. (1933), ``Der Massbegriff in der Theorie der kontinuierlichen Gruppen, \textit{Annals of Mathematics}, 2 34 (1), 147--167. JSTOR 1968346.

\bibitem{honda1999} Honda, A. and Okazaki, Y. (1999), Invariance of Fuzzy Measure, \textit{Bull. Kyushu Inst. Tech. Pure Appl. Math} 46, 1--8.

\bibitem{kacprzyk2017}
Kacprzyk, J., Szmidt, E., Zadrozny, S., Atanassov, K.T., Krawczak, M. (eds.). (2017).
\textit{Advances in Fuzzy Logic and Technology 2017: Proceedings of EUSFLAT-2017 and IWIFSGN'2017},
Warsaw, Poland, September 11-15, 2017. Advances in Intelligent Systems and Computing, vol. 642. Springer.

\bibitem{katsaras1977} Katsaras, A., D. B., L. (1977), \textit{Fuzzy vector spaces and fuzzy topological vector spaces}.
J. Math. Anal. Appl., 58, 135--146.

\bibitem{loomis1953} Loomis, L.H. (1953), \textit{An Introduction to Abstract Harmonic Analysis}, D. Van Nostrand Company Canada, Ltd, pp. 117--120.

\bibitem{nadjafikhah2009} Nadjafikhah, M., Bakhshandeh-Chamazkoti, R. (2009), \textit{Fuzzy Lie group},
arXiv preprint arXiv:0908.0254.


\bibitem{rosenfeld1971} Rosenfeld, A. (1971), Fuzzy groups, \textit{Journal of Mathematical Analysis and Applications} 35, 512--517.

\bibitem{sugeno1974} Sugeno, M. (1974), \textit{Theory of Fuzzy Integrals and its Applications}. Ph.D. dissertation, Tokyo Institute of Technology.

\bibitem{sugeno1977} Sugeno, M. (1977), ``Fuzzy measures and fuzzy integrals: A survey, In: Gupta, Saridis and Gaines [1977], 89--102.

\bibitem{wang1992} Wang, Z. (1992), \textit{Fuzzy Measure Theory}, Plenum Press, New York.

\bibitem{zadeh1996} Zadeh, L.A., Klir, G.J. and Yuan, B. (eds.) (1996),\textit{Fuzzy Sets, Fuzzy Logic, and Fuzzy Systems: Selected Papers},
World Scientific.

\bibitem{zhang1991} Zhang, G. (1991), \textit{Fuzzy Continuous Function and Its Properties}, Fuzzy Sets and Systems, 43(2), 159-171.

\end{thebibliography}
\end{document}